\documentclass[12pt]{amsart}
\usepackage{amssymb,latexsym}    
\usepackage{fullpage}

\input{xypic}
\newcommand{\Mdef}[2]{\newcommand{#1}{\relax \ifmmode #2 \else $#2$\fi}}

\newcommand{\sm }{\wedge}

\newcommand{\tensor}{\otimes}

\Mdef{\bhom}{\mathbf{\hat{H}om}}

\Mdef{\Mod}{\mathrm{mod}}

\newtheorem{thm}{Theorem}[section]
\newtheorem{lemma}[thm]{Lemma}
\newtheorem{prop}[thm]{Proposition}
\newtheorem{cor}[thm]{Corollary}

\theoremstyle{definition}

\newtheorem{example}[thm]{Example}

\newtheorem{remark}[thm]{Remark}

\newcommand{\qqed}{\qed \\[1ex]}
\renewenvironment{proof}[1][\hspace*{-.8ex}]{\noindent {\bf Proof #1:\;}}{\qqed}

\Mdef{\PH} {\Phi^H}
\Mdef{\PK} {\Phi^K}
\Mdef{\PL} {\Phi^L}
\Mdef{\PT} {\Phi^{\T}}

\Mdef{\ef}{E{\cF}_+}
\Mdef{\etf}{\widetilde{E}{\cF}}
\Mdef{\eg}{E{G}_+}
\Mdef{\etg}{\tilde{E}{G}}

\Mdef{\infl}{\mathrm{inf}}
\Mdef{\defl}{\mathrm{def}}
\Mdef{\res}{\mathrm{res}}
\Mdef{\ind}{\mathrm{ind}}
\Mdef{\coind}{\mathrm{coind}}

\Mdef{\univ}{\mathcal{U}}

\Mdef{\Fp}{\mathbb{F}_p}
\Mdef{\Zpinfty}{\Z /p^{\infty}}
\Mdef{\Zpadic}{\Z_p^{\wedge}}

\newcommand{\bi}{\begin{itemize}}
\newcommand{\be}{\begin{enumerate}}
\newcommand{\bc}{\begin{center}}
\newcommand{\bd}{\begin{description}}
\newcommand{\ei}{\end{itemize}}
\newcommand{\ee}{\end{enumerate}}
\newcommand{\ec}{\end{center}}
\newcommand{\ed}{\end{description}}

\newcommand{\adjunction}[4]{
\diagram
#1:#2 \rrto<0.7ex> &&
#3  \llto<0.7ex> :#4 
\enddiagram}
\newcommand{\lra}{\longrightarrow}
\newcommand{\lla}{\longleftarrow}

\Mdef{\we}{\mathbf{we}}
\Mdef{\fib}{\mathbf{fib}}
\Mdef{\cof}{\mathbf{cof}}
\Mdef{\BI}{\mathcal{BI}}

\newcommand{\ilim}{\mathop{ \mathop{\mathrm{lim}} \limits_\leftarrow} \nolimits}
\newcommand{\colim}{\mathop{  \mathop{\mathrm {lim}} \limits_\rightarrow} \nolimits}

\Mdef{\A}{\mathbb{A}}
\Mdef{\B}{\mathbb{B}}
\Mdef{\C}{\mathbb{C}}
\Mdef{\D}{\mathbb{D}}
\Mdef{\E}{\mathbb{E}}
\Mdef{\T}{\mathbb{T}}
\Mdef{\F}{\mathbb{F}}
\Mdef{\G}{\mathbb{G}}
\Mdef{\I}{\mathbb{I}}
\Mdef{\N}{\mathbb{N}}
\Mdef{\Q}{\mathbb{Q}}
\Mdef{\R}{\mathbb{R}}
\Mdef{\bbS}{\mathbb{S}}
\Mdef{\Z}{\mathbb{Z}}

\Mdef{\bA}{\mathbb{A}}
\Mdef{\bB}{\mathbb{B}}
\Mdef{\bC}{\mathbb{C}}
\Mdef{\bD}{\mathbb{D}}
\Mdef{\bE}{\mathbb{E}}
\Mdef{\bF}{\mathbb{F}}
\Mdef{\bG}{\mathbb{G}}
\Mdef{\bH}{\mathbb{H}}
\Mdef{\bI}{\mathbb{I}}
\Mdef{\bJ}{\mathbb{J}}
\Mdef{\bK}{\mathbb{K}}
\Mdef{\bL}{\mathbb{L}}
\Mdef{\bM}{\mathbb{M}}
\Mdef{\bN}{\mathbb{N}}
\Mdef{\bO}{\mathbb{O}}
\Mdef{\bP}{\mathbb{P}}
\Mdef{\bQ}{\mathbb{Q}}
\Mdef{\bR}{\mathbb{R}}
\Mdef{\bS}{\mathbb{S}}
\Mdef{\bT}{\mathbb{T}}
\Mdef{\bU}{\mathbb{U}}
\Mdef{\bV}{\mathbb{V}}
\Mdef{\bW}{\mathbb{W}}
\Mdef{\bX}{\mathbb{X}}
\Mdef{\bY}{\mathbb{Y}}
\Mdef{\bZ}{\mathbb{Z}}

\Mdef{\cA}{\mathcal{A}}
\Mdef{\cB}{\mathcal{B}}
\Mdef{\cC}{\mathcal{C}}
\Mdef{\mcD}{\mathcal{D}} 
\Mdef{\cE}{\mathcal{E}}
\Mdef{\cF}{\mathcal{F}}
\Mdef{\cG}{\mathcal{G}}
\Mdef{\mcH}{\mathcal{H}} 
\Mdef{\cI}{\mathcal{I}}
\Mdef{\cJ}{\mathcal{J}}
\Mdef{\cK}{\mathcal{K}}
\Mdef{\mcL}{\mathcal{L}}

\Mdef{\cM}{\mathcal{M}}
\Mdef{\cN}{\mathcal{N}}
\Mdef{\cO}{\mathcal{O}}
\Mdef{\cP}{\mathcal{P}}
\Mdef{\cQ}{\mathcal{Q}}
\Mdef{\mcR}{\mathcal{R}}
\Mdef{\cS}{\mathcal{S}}
\Mdef{\cT}{\mathcal{T}}
\Mdef{\cU}{\mathcal{U}}
\Mdef{\cV}{\mathcal{V}}
\Mdef{\cW}{\mathcal{W}}
\Mdef{\cX}{\mathcal{X}}
\Mdef{\cY}{\mathcal{Y}}
\Mdef{\cZ}{\mathcal{Z}}

\Mdef{\At}{\tilde{A}}
\Mdef{\Bt}{\tilde{B}}
\Mdef{\Ct}{\tilde{C}}
\Mdef{\Et}{\tilde{E}}
\Mdef{\Ht}{\tilde{H}}
\Mdef{\Kt}{\tilde{K}}
\Mdef{\Lt}{\tilde{L}}
\Mdef{\Mt}{\tilde{M}}
\Mdef{\Nt}{\tilde{N}}
\Mdef{\Pt}{\tilde{P}}

\Mdef{\tA}{\tilde{A}}
\Mdef{\tB}{\tilde{B}}
\Mdef{\tC}{\tilde{C}}
\Mdef{\tE}{\tilde{E}}
\Mdef{\tH}{\tilde{H}}
\Mdef{\tK}{\tilde{K}}
\Mdef{\tL}{\tilde{L}}
\Mdef{\tM}{\tilde{M}}
\Mdef{\tN}{\tilde{N}}
\Mdef{\tP}{\tilde{P}}

\Mdef{\ft}{\tilde{f}}
\Mdef{\xt}{\tilde{x}}
\Mdef{\yt}{\tilde{y}}

\Mdef{\Ab}{\overline{A}}
\Mdef{\Bb}{\overline{B}}
\Mdef{\Cb}{\overline{C}}
\Mdef{\Db}{\overline{D}}
\Mdef{\Eb}{\overline{E}}
\Mdef{\Fb}{\overline{F}}
\Mdef{\Gb}{\overline{G}}
\Mdef{\Hb}{\overline{H}}
\Mdef{\Ib}{\overline{I}}
\Mdef{\Jb}{\overline{J}}
\Mdef{\Kb}{\overline{K}}
\Mdef{\Lb}{\overline{L}}
\Mdef{\Mb}{\overline{M}}
\Mdef{\Nb}{\overline{N}}
\Mdef{\Ob}{\overline{O}}
\Mdef{\Pb}{\overline{P}}
\Mdef{\Qb}{\overline{Q}}
\Mdef{\Rb}{\overline{R}}
\Mdef{\Sb}{\overline{S}}
\Mdef{\Tb}{\overline{T}}
\Mdef{\Ub}{\overline{U}}
\Mdef{\Vb}{\overline{V}}
\Mdef{\Wb}{\overline{W}}
\Mdef{\Xb}{\overline{X}}
\Mdef{\Yb}{\overline{Y}}
\Mdef{\Zb}{\overline{Z}}

\Mdef{\db}{\overline{d}}
\Mdef{\hb}{\overline{h}}
\Mdef{\qb}{\overline{q}}
\Mdef{\rb}{\overline{r}}
\Mdef{\tb}{\overline{t}}
\Mdef{\ub}{\overline{u}}
\Mdef{\vb}{\overline{v}}

\Mdef{\hc}{\hat{c}}
\Mdef{\he}{\hat{e}}
\Mdef{\hf}{\hat{f}}
\Mdef{\hA}{\hat{A}}
\Mdef{\hH}{\hat{H}}
\Mdef{\hJ}{\hat{J}}
\Mdef{\hM}{\hat{M}}
\Mdef{\hP}{\hat{P}}
\Mdef{\hQ}{\hat{Q}}

\Mdef{\thetab}{\overline{\theta}}
\Mdef{\phib}{\overline{\phi}}

\Mdef{\uA}{\underline{A}}
\Mdef{\uB}{\underline{B}}
\Mdef{\uC}{\underline{C}}
\Mdef{\uD}{\underline{D}}

\Mdef{\bolda}{\mathbf{a}}
\Mdef{\boldb}{\mathbf{b}}
\Mdef{\boldD}{\mathbf{D}}

\Mdef{\fm}{\frak{m}}
\Mdef{\fp}{\frak{p}}

\Mdef{\eps}{\epsilon}

\newcommand{\mccM}{\mathcal{M}}

\renewcommand{\Et}{\cE_t}

\newcommand{\dcell}{\mbox{-cell-}}
\newcommand{\modcat}[1]{\mbox{$#1$-mod}}

\newcommand{\mccR}{R}

\newcommand{\Categ}{\mathrm{Cat}}

\newcommand{\ccI}{\bfD}
\newcommand{\ccIp}{\bfD_+}
\newcommand{\bbarmccR}{\hat{\mccR}}
\newcommand{\fmccR}{f{\mccR}}
\newcommand{\fRp}{f{\mccR}_+}

\newcommand{\bfD}{\mathbf{D}}

\newcommand{\bfE}{\mathbf{E}}

\newcommand{\MD}{\mccM|_{\bfD}}

\begin{document}
\title{Homotopy theory of modules over diagrams of rings} 
\author{J.~P.~C.~Greenlees}
\address{Department of Pure Mathematics, The Hicks Building, 
Sheffield S3 7RH. UK.}
\email{j.greenlees@sheffield.ac.uk}

\author{B.~Shipley}
\thanks{The first author is grateful for support under EPSRC grant
  number EP/H040692/1.   This material is based upon work by the second author supported by the National Science Foundation under Grant No. DMS-1104396.
}
\address{Department of Mathematics, Statistics and Computer Science, University of Illinois at
Chicago, 508 SEO m/c 249,
851 S. Morgan Street,
Chicago, IL, 60607-7045, USA}
\email{bshipley@math.uic.edu}

\date{}
\maketitle

\begin{abstract}
Given a diagram of rings, one may consider the category of modules
over them. We are interested in the homotopy theory of categories of
this type: given a suitable diagram of model categories $\mccM(s)$ (as
$s$ runs through the diagram), we consider the category of diagrams
where the object $X(s)$ at $s$ comes from $\mccM(s)$. We develop model structures on such categories
of diagrams,  and Quillen adjunctions that relate categories based on
different diagram shapes.  

Under certain conditions, cellularizations (or right Bousfield localizations) of these adjunctions induce Quillen equivalences.  As an application we show that a cellularization of a category of modules over a diagram of ring spectra (or differential graded rings) is Quillen equivalent to modules over the associated inverse limit of the rings.  Another application of the general machinery here is given in work by the authors on algebraic models of rational equivariant spectra.  Some of this material originally appeared in the preprint ``An algebraic model for rational torus-equivariant stable homotopy theory", arXiv:1101.2511,  but has been generalized here.  
%
%
\end{abstract}

\section{Introduction}
Given a diagram of rings, one may consider the category of modules
over them. We are interested in the homotopy theory of categories of
this type: given a suitable diagram of model categories $\mccM(s)$
where functors relating them are left Quillen functors,
 we consider the category of diagrams where the object $X(s)$ at $s$ comes
from $\mccM(s)$. The purpose of this paper is to show that under suitable hypotheses,  there are
diagram-projective and diagram-injective model structures on the
category (Theorem \ref{thm.injective.model}), and
to investigate Quillen adjunctions associated to restricting the
diagram (Theorems \ref{thm5.3} and \ref{thm5.4}). 

\subsection{Motivation}
This paper grew out of our project on algebraic models for rational
$G$-equivariant spectra for $G$ a torus~\cite{tnq3}.  The main result
of that project 
is to show that the homotopy theory of rational $G$-spectra is
modeled by an algebraic category of diagrams, and it is worth describing the
strategy to illustrate the use of the techniques developed in the
present paper. 
We begin by showing that the category of $G$-spectra is modelled by a
diagram of modules  over equivariant ring spectra and in the end
we show that this is modelled by a category of diagrams of
differential graded modules over  graded rings.    Some of the necessary generality is
slightly hidden here, since in the spectral part we must consider a context where
not only the ring, but also the group of equivariance varies with the
position in the diagram. 
With this generality, which motivates the setting of this current paper, we are able to describe the various models we use. 

The next issue is that the shape of the diagram of equivariant spectra 
we start with  is different from the shape of the diagram of
differential graded rings we end with.
To relate categories based on these  two diagram shapes we construct a
larger diagram shape category which contains the smaller ones, so that a
diagram  based on the new larger shape restricts to two smaller diagrams 
of the original shapes. We then 
show that a suitable inclusion of diagram shapes induces a Quillen
adjunction, and apply  the Cellularization Principle \cite{GScell} to show that it induces a 
Quillen equivalence after cellularization. 
 
Since the techniques of using diagrams of categories and of changing
the diagram shapes can be generalized and applied in other settings, 
we have decided to present  it separately here in appropriate
generality and we refer to~\cite{tnq3new} for the original application.

\subsection{Organization}
In Sections \ref{sec:ringsandmodules} and \ref{sec:model}, we develop
model structures for categories of direct or inverse diagrams where
the model category from which the objects come
varies with the position in the diagram. In 
Section \ref{sec-inverse}, as an example, we consider diagrams of
modules over a diagram of ring spectra (or differential graded
rings). A particularly well-known example is (differential graded)
modules over the classical Hasse square which considers the integers
as the pull-back ring of the rationals, the $p$-adic integers for all
primes $p$, and their tensor products.  Using the Cellularization
Principle \cite{GScell} (see also Appendix~\ref{app:cell}), we show
that modules over the homotopy inverse limit of a given diagram of
rings can be modelled by the cellularization of the category of modules
over the diagram of rings (Proposition \ref{prop-gen-pb}).  This is
the model category version of the local to global principle for the Hasse square.  

In Section \ref{sec:adjoints}, we consider changing diagram shapes.  In particular, we consider the inclusion of a subcategory $i: \bfD \to \bfE$ and its induced restriction functor on diagram categories over $\bfD$ and $\bfE$.  We then show that after certain cellularizations (or co-localizations) these different shaped diagram categories model the same homotopy theory.   At the end of Section \ref{sec:adjoints}, we return to the inverse limit example of Section \ref{sec-inverse} to show that it is an example of this general machinery for changing diagram shapes.

\section{Diagrams of rings and modules}
\label{sec:ringsandmodules}

Categories of modules over diagrams of
rings have created useful new models; see for example~\cite{ss2, tnq3}.   These examples use two underlying contexts:
differential graded modules over differential graded algebras (DGAs)  and module spectra over ring spectra. 
In~\cite{tnq3}, we  needed to generalize this setting further to work with equivariant spectra.
   
\subsection{The archetype}
\label{sec:archetype}
Given a diagram shape $\bfD$, consider a diagram of rings $R:\bfD \lra \C$ in a symmetric monoidal category $\C$.  
Each map $R(a): R(s)\lra R(t)$ gives rise to an extension of scalars functor
$$ \modcat{R(s)}\stackrel{a_{*}}\lra \modcat{R(t)} $$
defined by $a_*(X)=R(t)\tensor_{R(s)}X$, 
 with right adjoint the restriction of scalars functor
$$ \modcat{R(s)}\stackrel{a^*}\lla \modcat{R(t)}. $$

Now consider a category of {\em $R$-modules}; 
the objects are diagrams $X: \bfD \lra \C$ for which $X(s)$ is an
$R(s)$-module for each object $s$, 
and for every morphism $a: s\lra t$ in $\bfD$, the map $X(a): X(s) \lra X(t)$ 
is a {\em module map over the ring map} $R(a): R(s) \lra R(t)$. More
precisely, there is a map $X(s) \lra a^*X(t)$ of $R(s)$-modules (the {\em
 restriction} model)  or,  equivalently, there is a map
$$R(t)\tensor_{R(s)} X(s) =a_*X(s) \lra X(t) $$
of modules over the ring $R(t)$ (the {\em extension of  scalars}
model).  Although restriction of scalars seems very simple,  in the
more general case, it is more natural to view the left adjoint
$a_*$ as the primary one.

\subsection{A generalization}
\label{subsec:modcatdiag}
We next consider a 
generalization of this archetype. Here we begin with a diagram 
$\mccM: \bfD \lra \Categ$ of categories.
The previous special case is $\mccM(s)=\modcat{R(s)}$
and in our applications each category $\mccM(s)$ is a  category of
modules in some  category $\C(s)$ which also varies with $s$. 
Since $\mccM$ is a functor, for each $a:s\lra t$ in $\bfD$ we have an
associated functor $a_* :\mccM (s) \lra \mccM (t)$ which is compatible
with composition in $\bfD$.   
We then consider the category of $\mccM$-diagrams, $\modcat{\mccM}$.  The objects in this category consist of an 
object $X(s)$ from $\mccM (s)$ for each object $s$ of $\bfD$ with a transitive system of morphisms
$$\widetilde{X}(a): a_*X(s) \lra X(t)$$
for each morphism $a:s\lra t$ in $\bfD$
(the {\em left adjoint} form). 
If  each $a_*$ has a right adjoint $a^*$, then the system of morphisms is equivalent to an adjoint system of morphisms
$$\hat{X}(a): X(s) \lra a^*X(t)$$ for each morphism $a:s\lra t$ in $\bfD$
(the {\em right adjoint} form).

\subsection{Model structures}\label{subsec:modstructure}
We say that $\mccM$ is {\em a diagram of model categories} if each category
$\mccM (s)$ has a model structure, the functors $a_*$ all have right
adjoints and  the adjoint pair $a_*\vdash a^*$ of functors relating the 
model categories form a Quillen pair. 

For instance, the motivating example of a diagram of rings gives a diagram of model
categories if we use the projective model structure on the category 
$\mccM (s)$ of  $R(s)$-modules.

When $\mccM$ is a diagram of model categories, there are two ways to attempt to put a model structure on 
the category of $\mccM$-diagrams $\{ X(s)\}_{s \in \bfD}$. The
{\em diagram-projective} model structure (if it exists) has its fibrations and weak  
equivalences defined  objectwise. The {\em diagram-injective} model
structure  (if it exists) has its cofibrations and 
weak  equivalences defined objectwise.  It must be checked in each 
particular case whether or not these specifications determine a model
structure. When both model structures exist, it is clear that the identity functors
define a Quillen equivalence between them.

We will prove Theorem \ref{thm.injective.model} stating that 
the diagram-projective and diagram-injective model structures exist for certain diagram shapes $\bfD$.

\subsection{Simple change of diagrams}
\label{subsec:elemequivalences}
Returning to the archetype diagram of modules over a diagram of rings, 
sometimes if we specify the modules
on just part of the diagram we can fill in the remaining entries
using adjoints. There are two types of examples: 
(1) the diagram is filled in by using left adjoints such as extension of scalars and direct limits and (2) the diagram is filled in using right adjoints such as restriction and inverse limits.  In both cases, this sometimes induces a Quillen equivalence 
between  subcategories of modules over the larger and smaller diagrams.  
In Section~\ref{sec-inverse}, we develop an example of type (2) of a Quillen equivalence of module categories.  
In Section~\ref{sec:adjoints}, we develop general statements for both types (1) and (2) in the setting of diagrams of module categories as in Section~\ref{subsec:modcatdiag}.

\begin{example}\label{ex1}
The simplest example of Type (1) starts with a diagram $R=(R_0\lra R_1)$ of
rings. An $R$-module gives rise to an $R_0$-module by evaluation at
the first object. An $R_0$-module $X_0$ produces the $R$-module
$$\left( R\tensor_{R_0}X_0 \right)
=\left( R_0\tensor_{R_0}X_0\lra R_1\tensor_{R_0}X_0 \right) .$$ 
\end{example}

\begin{example}\label{ex2}
The simplest example of Type (2) starts with a diagram $R=(R_0\lra
R_{01}\lla R_1)$ of rings. If we let $\hat{R}$ be defined by  the pullback
$$\diagram
\hat{R}\rto \dto&  R_0\dto \\
R_1 \rto & R_{01}
\enddiagram$$
  an $R$-module gives rise to an $\hat{R}$-module by 
pullback. An $\hat{R}$-module $\hat{X}$ produces the $R$-module
$$\diagram
&  R_0\tensor_{\hat{R}}\hat{X}\dto \\
R_1\tensor_{\hat{R}}\hat{X} \rto & R_{01}\tensor_{\hat{R}} \hat{X} . 
\enddiagram$$
\end{example}

Returning to the general case in more detail, we let $i: \bfD \lra \bfE$ be the inclusion of 
a subdiagram, and $R: \bfE \lra \C$ be a diagram of
rings.  We restrict $R$ to  a diagram $R|_{\bfD}: \bfD \lra \C$,
and this induces a restriction functor 
$$i^*: \modcat{R}\lra \modcat{R|_{\bfD}}. $$
We discuss two cases in detail in Section~\ref{sec:adjoints}, depending on whether we focus on $i^*$ as a right or left
adjoint. 

\vspace{1ex}
\noindent
{\em The Left Adjoint Case.}
If $i^*$ has a  left adjoint $i_*$ and we consider diagram-projective model structures (with objectwise weak equivalences and fibrations) on the two categories, then
 the adjunction $(i_*, i^*)$  is a Quillen pair.  

In fact for a diagram $M$ on $\bfD$, we may identify $i_*M$
explicitly. To find its value at an object $t$ of $\bfE$ we consider the category $\bfD /t$
whose objects are morphisms $s\lra t$ in $\bfE$ with $s$ in $\bfD$ and then take
$i_*M(t)=\colim_{s\in \bfD/t} a_* M(s)$ (closely related to the latching
object at $t$).  In particular, if objects of $\bfD$ have no automorphisms  
and $s$ is in $\bfD$, then $\mbox{id}_s$ is a terminal object of $\bfD /s$ and $i_*$ will not
change the value at $s$. In this case,  $i_*$ leaves the entries in
$\bfD$  unchanged and  the  unit $M\lra i^*i_*M$ is an isomorphism.

\vspace{1ex}
\noindent
{\em The Right Adjoint Case.}
Similarly, if $i^*$ has a right adjoint, $i_!$, and we consider
diagram-injective model structures (with objectwise weak equivalences and cofibrations) then the adjunction
$(i^*, i_!)$ is a Quillen pair. 

In fact for a diagram $M$ on $\bfD$, we may identify $i_!M$
explicitly. To find its value at an object $t$ of $\bfE$,  we consider the category $t /  \bfD$
whose objects are morphisms $t\lra s$ in $\bfE$ with $s$ in $\bfD$, and then take
$i_!M(s)=\ilim_{s\in t/ \bfD} a^* M(s)$ (closely related to the
matching object at $t$).  In particular, if objects of $\bfD$ have no automorphisms 
and $s$ is in $\bfD$ then $\mbox{id}_s$ is an initial object of $s / \bfD$ and $i_!$ will not
change the value at $s$. In this case,  $i_!$ leaves the entries in
$\bfD$  unchanged  and the  counit $i^*i_!M\lra M$ is an  isomorphism.

\section{Diagram-injective model structures}\label{sec:model}

In this section we develop {\em diagram-projective} and  {\em diagram-injective} model structures 
for the generalized categories of diagrams defined in Section~\ref{subsec:modcatdiag}.   As we see in Remark~\ref{rem:many.objects} below, one familiar example of such a generalized category of diagrams is the category of 
modules over a ring with many objects, \cite[3.3.2]{ss2}.  
In that case, the {\em diagram-projective} (or {\em standard})
model structure here agrees with the one developed in \cite[A.1.1]{ss2} and has objectwise
weak equivalences and fibrations; see also \cite{A}. 
In contrast, the diagram-injective model structure here has weak equivalences and cofibrations determined at each object.
These are the analogues of the model structures for diagrams over direct and inverse
small categories developed, for example, in \cite[5.1.3]{hovey-model}.  

We restrict our attention here to the diagrams indexed on
small direct (or inverse) categories.  
Let $\ccI$ be a small direct category 
with a fixed linear extension $d: \ccI \to \lambda$ for some ordinal $\lambda$. 
Note that if $\ccI(s,t)$ is non-empty and $s \not = t$  then $d(s) < d(t).$
Let $\mccM$ be a diagram of model categories indexed by $\ccI$; that is, 
each $s \in \ccI$ is assigned a model category $\mccM(s)$ and each $a: s \to t$ in $\ccI$ is assigned
a left Quillen functor $a_{*}: \mccM(s) \to \mccM(t)$ (with right adjoint $a^{*}$) which are compatible with composition. 
Then a diagram $X$ over $\mccM$ (or ``$\mccM$-diagram") specifies
 for each object $s$ in $\ccI$ an object $X(s)$ of $\mccM(s)$ and for each morphism $a: s \to t$ in
 $\ccI$ a map $\widetilde{X}(a): a_{*}X(s) \to X(t)$, again compatible with compositions.
Let $\ccI_{t}$ be the category whose objects are all non-identity maps  in $\ccI$ with codomain $t$.
Then any diagram $X$ induces a functor 
from $\ccI_{t}$ to $\mccM(t)$ by
taking $a:s \to t$ in $\ccI_{t}$ to $a_{*}X(s)$.
Define the {\em latching space} functor, $L_t X$ as the direct limit in $\mccM(t)$,
\[ L_t X = \colim_{\ccI_{t}} a_{*}X(s).\]

In the dual situation where $\ccI$ is a small inverse category, we consider 
again a diagram of model categories $\mccM$.
Note, here again each $a: s \to t$ in $\ccI$ is assigned to
a left Quillen functor $a_{*}: \mccM(s) \to \mccM(t)$ with right adjoint $a^{*}$.
Let $\ccI^{s}$ be the category of all non-identity maps  in $\ccI$ with domain $s$.
Then any $\mccM$ diagram $X$ induces a functor from $\ccI^{s}$ to $\mccM(s)$ by taking $a:s \to t$ in $\ccI^{s}$ to $a^{*}X(t)$.
Define the {\em matching space} functor, $M_i X$ as the inverse limit in $\mccM(i)$,
\[ M_s X = \ilim_{\ccI^{s}} a^{*}X(t).\]

\begin{thm} \label{thm.injective.model}
Assume given a category $\ccI$ and a diagram of model categories, $\mccM$, indexed
on $\ccI$ as above.   

(i) If $\ccI$ is a direct category, then there is a {\em diagram-projective} model structure on the category of diagrams over $\mccM$ with
objectwise weak equivalences and fibrations; that is,
$X \to Y$ is a weak equivalence (or fibration) if $X(s) \to Y(s)$ is an
underlying weak equivalence (or fibration) in $\mccM(s)$ for all $s$.    This map is a
(trivial) cofibration if and only if the induced map $X(s)\coprod_{L_{s}X} L_{s}Y \to Y(s)$ is a (trivial) cofibration in $\mccM(s)$ for all $s$.

(ii) If $\ccI$ is an inverse category, then there is
a {\em diagram-injective} model structure on the category of diagrams over $\mccM$ with
objectwise weak equivalences and cofibrations; that is,
$X \to Y$ is a weak equivalence (or cofibration) if $X(s) \to Y(s)$ is an
underlying weak equivalence (or cofibration) in $\mccM(s)$ for all $s$.    This map is a
(trivial) fibration if and only if the induced map $X(s)\ \to Y(s) \times_{M_{s}Y} M_{s}X$ is a (trivial) fibration in $\mccM(s)$ for all $s$.
\end{thm}

\begin{proof}
The verification of the axioms follows the same outline as in \cite[5.1.3]{hovey-model}.
The only difference is that here the ambient category changes at each object
in $\ccI$.     Instead of repeating these arguments, we give some of the details for these changing categories.  As in \cite[5.1.3]{hovey-model} we consider only the direct category case,
since the inverse category case is dual.

Define $\ccI_{<\beta}$ as the full subcategory of $\ccI$ on all objects $i$ such that
$d(s) < \beta$.   Then let $\mccM_{<\beta}$ denote the diagram of model categories induced by
the restriction of $\mccM$ to $\ccI_{<\beta}$.   Similarly, for any $\mccM$ diagram $X$, the restriction 
to $\ccI_{<\beta}$ gives an $\mccM_{<\beta}$ diagram $X_{<\beta}$.   
Given these definitions, the lifting axioms follow by induction as in 
 \cite[5.1.4]{hovey-model}, by producing
lifts for the various restrictions to $\mccM_{<\beta}$ diagrams.   Note that at
the successor ordinal case the relevant commutative diagram is just a usual lifting
problem in $\mccM(s)$.

To complete the verification of the model structure we follow the proof of \cite[5.1.3]{hovey-model}.
That proof uses \cite[5.1.5]{hovey-model} to consider maps formed by colimits. In the usual setting, the colimit is the left adjoint to the constant functor.  Here though instead of the constant functor one must use the relevant right adjoint. The colimit in question in our analogue of the proof of \cite[5.1.3]{hovey-model} is the functor $L_s$; denote its right adjoint by $G_{s}$. For an object $X$ in $\mccM(s)$,
the $\mccM$ diagram $G_{s}X$ at $a \in \ccI_{s}$ is $a^{*}X$.  Since each $a^{*}$ is a right 
Quillen functor, $G_{s}$ takes (trivial) fibrations to objectwise (trivial) fibrations. Thus the required analogue of \cite[5.1.5]{hovey-model} holds in our setting as well.  

The only other change needed in the proof of \cite[5.1.3]{hovey-model} is that for the induction 
step in the construction of the functorial factorizations one uses factorization 
in $\mccM(s)$ to factor the map $X(s)\coprod_{L_{s}X} L_{s}Y \to Y(s)$.
\end{proof}

\begin{remark}\label{rem.johnson}
In~\cite{johnson} Reedy diagrams are considered where the model structure
is allowed to vary although the underlying category does not vary.  One could further generalize Johnson's results to allow the underlying category to vary.  The necessary added conditions (see~\cite[3.1, 5.1]{johnson}) would require that the functor associated to any arrow in the diagram was a left Quillen functor. Since we do not have any such examples for motivation, we have not pursued this generalization.
\end{remark}

For the applications in~\cite{tnq3new}, we only need to consider diagram categories $\ccI$ with at most one map between any two objects.   Restricting to this situation simplifies the arguments for the following proposition.

\begin{prop}\label{prop-injective-model-properties}
Let $\ccI$ be a direct (or inverse) category with at most one map between any two
objects.  Assume given a diagram of proper, cellular model categories $\mccM$; that is, for each $s \in \ccI$, the model structure $\mccM(s)$ is proper
and cellular.
Then the diagram-projective (or diagram-injective) model structure on $\mccM$-diagrams defined in Theorem~\ref{thm.injective.model}
is a proper, cellular model category.  
\end{prop}

\begin{proof}
We first establish properness.  In the diagram-projective case fibrations
and weak equivalences are defined objectwise and one can show that any cofibration induces
an objectwise cofibration.   Since pullbacks and pushouts are constructed at each object
and $\mccM(s)$ is assumed to be a proper model structure for each $s$, 
properness follows.  The diagram-injective case is dual.

We use Hirschhorn's treatment of Reedy categories~\cite[Chapter 15]{hh}
to establish that these model structures are cellular.  Note that a direct category is an example of a Reedy category with no morphisms that lower 
degrees.   In this case, the matching categories are empty so that the matching objects are just the terminal object.  Thus, the Reedy fibrations are just the objectwise fibrations and the 
Reedy model structure~\cite[15.3.4]{hh} agrees with the diagram-projective model structure defined above.   The arguments for an inverse category are dual.

Next we define the generating (trivial) cofibrations.   
Given an object $A$ in $\mccM(s)$, define the free $\mccM$ diagram generated by 
$A$ at $s$ to be
$\cF^{s}_{A}(t) = a_{*}A$ when $\ccI(s,t)=\{a\}$ is non-empty and the
initial object otherwise.   
For $\ccI$ a direct category, and $f:A\to B$ in $\mccM(s)$, define
$RF^{s}_{f}$ to be the induced map of diagrams $\cF^s_A \to \cF^s_B$.
 Let $I_{s}$ denote the generating cofibrations for $\mccM(s)$.  Let $RF^{\ccI}_{I}$ denote the set 
 of maps $RF^{s}_{f}$ for all maps $f$ in $I_{s}$ for all $s$ in $\ccI$.   Define
$RF^{\ccI}_{J}$ similarly based on the sets $J_{s}$ of generating trivial cofibrations
for $\mccM(s)$.   
 By~\cite[15.6.27]{hh},  the diagram-projective model structure on 
 $\mccM$ diagrams is cofibrantly generated with generating cofibrations $RF^{\ccI}_{I}$ and generating trivial cofibrations $RF^{\ccI}_{J}$.  

For $\ccI$ an inverse category, we define the boundary of the free functor $\cF^s_X$ to be
$\partial \cF^{s}_{X}(t) = a_{*}X$ when $\ccI(s,t) = \{a\}$ is non-empty and $s \ne t$ and the initial object otherwise.  Note that these functors only differ at $s=t$. Given a map $f: A \to B$ in
$\mccM(s)$, let $RF^{s}_{f}$ denote the $\mccM$ diagram map 
\[ \cF^{s}_{A} \coprod_{\partial \cF^{s}_{A}}\partial \cF^{s}_{B} \to \cF^{s}_{B}.
\]
(In~\cite[15.6.18]{hh}, the boundary $\partial \cF^{s}$ of a free functor is defined for general Reedy 
categories and uses the non-identity maps in $\ccI$ which lower degree.   For $\ccI$ a direct category, this simplifies since no map lowers degree. Thus, $\partial \cF^{s}$ is just the 
diagram of initial objects and one recovers the above definition of $RF^s_f$.)   
As in the case of a direct category, by~\cite[15.6.27]{hh}, the diagram-injective model structure on $\mccM$ diagrams is cofibrantly generated with generating cofibrations $RF^{\ccI}_I$ and generating trivial cofibrations $RF^{\ccI}_J$ defined as above using the generating cofibrations $I_s$ and trivial cofibrations  $J_s$ from $\mccM(s)$.  
 
Finally, \cite[15.7.6]{hh} establishes the additional conditions for showing this is a cellular model category given that each category $\mccM(s)$ is a cellular model category.
\end{proof}

\begin{remark}\label{rem:many.objects}
For $\ccI$  an inverse category with at most one morphism in each $\ccI(s,t)$, modules over a diagram of rings over $\C$
are equivalent to
categories of modules over a ring with many objects over $\C$.
If $\ccI(s,t)$ is non-empty then the map $\mccR(s) \to \mccR(t)$ makes
$\mccR(t)$ an $\mccR(s)$-module.   
There is an associated $\C$-enriched category indexed on the objects of $\ccI$ which we also
denote by $\mccR$ with $\mccR(s,s)$ the
ring $\mccR(s)$, $\mccR(s,t)$ trivial when $\ccI(s,t)$ is empty, and 
$\mccR(s,t)$ the $\mccR(t)$ - $\mccR(s)$ 
bimodule $\mccR(t)$ when $\ccI(s,t)$ is non-empty.

A (left) module $M$ over $\mccR$ is a covariant $\C$-enriched functor from $\mccR$ to
$\C$.   The data needed to specify such a module is exactly the same as given
for a module over the associated diagram $\mccR$.   First, for each object $s$ in $\ccI$,
$M(s)$ is an $R(s,s) (= R(s))$ module and for each morphism $a: s \to t$ in $\ccI$
the module structure specifies a map $R(s,t) \otimes_{R(s,s)} M(s) \to M(t)$.   Since $R(s,t) = R(t)$,
this is the required map $F_{a}M(s) \to M(t)$ where $F_{a}$ is extension of scalars over
$R(s) \to R(t)$.   We consider covariant functors here because this eases the comparison with diagrams even though this differs from the right modules (or contravariant functors) considered 
in~\cite[3.3.2]{ss2}.   
\end{remark}

\section{Inverse limit example}\label{sec-inverse}

In this section we develop a result comparing modules over a diagram of rings and modules over the homotopy inverse limit of the diagram of rings.  We show that the adjunction associated to the change from the diagram of rings to the one homotopy inverse limit ring induces a Quillen equivalence after applying the Cellularization Principle from~\cite{GScell}; see also Proposition~\ref{prop-cell-qe}.  This is a model for the more general adjunctions considered in Section~\ref{sec:adjoints}.   The particular case of a pull back diagram of rings, such as in the classical Hasse principle, is treated in more detail in Section 6 of~\cite{GScell}.   Here we will work in the context of ring and module spectra, but this material also easily translates to the differential graded context.  Note though that it is necessary to be in a stable context to use Proposition~\ref{prop-cell-qe} and here $A$-cellularization denotes cellularization with respect to all suspensions and desuspensions of $A$.
   
Assume given a finite, inverse category $\ccI$ with at most one morphism in each $\ccI(s,t)$ and a diagram of ring spectra, $\mccR$, indexed on 
$\ccI$.   We consider the associated diagram of model categories $\mccM$ with $\mccM(s)$ the model category of $R(s)$-module spectra and $F_{a} = R(t) \sm_{R(s)} (-)$
the left Quillen functor given by extension of scalars.  We refer to $\mccM$-diagrams
as $\mccR$-modules and compare the diagram-injective model category of $\mccR$-modules with modules over the homotopy inverse limit 
of the diagram $\mccR$.   

By~\cite[19.9.1]{hh},
the homotopy inverse limit of $\mccR$ is the inverse limit of a fibrant replacement of $\mccR$ in the diagram-injective model category of 
$\ccI$-diagrams of ring spectra.   This model structure exists
by~\cite[5.1.3]{hovey-model}; see also~\cite[15.3.4]{hh} since an
inverse category is a particular example of a Reedy category.  Let $g:
\mccR \to \fmccR$ be this fibrant 
replacement and let $\hat{\mccR}$
denote the inverse limit over $\ccI$ of $f\mccR$.  
We compare $\mccR$-modules and $\bbarmccR$-modules via the
category of $\fmccR$-modules.  Since $\mccR \to \fmccR$ is an objectwise weak equivalence, there is a Quillen equivalence between $\mccR$-modules and $\fmccR$-modules by 
Lemma~\ref{lem.ext.scalars} below.
We also establish below  a Quillen adjunction between ${\bbarmccR}$-modules and $\fmccR$-modules
which is a Quillen equivalence after cellularization.  To satisfy the smallness hypotheses needed in the Cellularization Principle~\ref{prop-cell-qe}, we must assume that $\ccI$ is a finite category. This leads to the following statement.

\begin{prop}\label{prop-gen-pb}
For $\ccI$ a finite, inverse category with at most one morphism in each $\ccI(s,t)$ and $R$ a $\ccI$-diagram of ring spectra with homotopy inverse limit $\bbarmccR$, there is a zig-zag of Quillen equivalences between the category of $\bbarmccR$-modules and the cellularization with respect to $\mccR$ of $\mccR$-modules.
\[ \modcat{\bbarmccR} \simeq_{Q} \mbox{$\mccR$-cell-}\modcat{\mccR} \]
\end{prop}

We first need the following lemma.   

\begin{lemma} \label{lem.ext.scalars}
Assume given $L: \mccM \to \cN$ a map of diagrams of model categories
over a direct category $\ccI$.   
If each $L(s): \mccM(s) \to \cN(s)$ is a left Quillen equivalence, then $L$ 
induces a Quillen equivalence between the  diagram-projective
model structures of $\mccM$ and $\cN$
diagrams.
If $\ccI$ is instead an inverse category, then $L$ induces a Quillen equivalence between the diagram-injective model structures. 
\end{lemma}

\begin{proof}
Let $R$ denote the right adjoint of $L$.  Since each $R(s)$ is a right Quillen functor, $R$ preserves  the objectwise fibrations and weak equivalences of the diagram-projective model structures.
The equivalence then follows since a cofibrant or fibrant $\mccM$ diagram is objectwise 
cofibrant or fibrant in the diagram-projective model structure.
Namely, given a cofibrant $\mccM$ diagram $X$ and a fibrant $\cN$ diagram $Y$,
a map $LX \to Y$ is an objectwise weak equivalence if and only if 
$X \to RY$ is an objectwise weak equivalence since $L$ and its right adjoint $R$
are objectwise Quillen equivalences.   The diagram-injective case is similar.
\end{proof}

\begin{proof}[of Proposition~\ref{prop-gen-pb}] 
Since $g: \mccR \to \fmccR$ is an objectwise weak equivalence, and extension of
scalars along weak equivalences
of ring spectra induce Quillen equivalences,
 the associated diagram module
categories are Quillen equivalent by Lemma~\ref{lem.ext.scalars}.   Once we verify that $\mccR$ is small (and hence also $\fmccR$ is small), Corollary~\ref{cor-cell-qe} shows that this induces a Quillen equivalence on the cellularizations of the diagram-injective model
structures
\[ \mbox{$\mccR$-cell-}\modcat{\mccR}  \simeq_{Q}   \mbox{$\fmccR$-cell-}\modcat{\fmccR}\]
since $\mccR$ is cofibrant in $\mccR$-modules and extension of scalars takes
$\mccR$ to $\fmccR$.

To show that the object $\mccR$ is small in $\mccR$-modules, we first show it is the finite colimit of small objects in $\mccR$-modules. Let $L^s$ denote the left adjoint to evaluation at the object $s \in \ccI$.  Note that the objects $L^sR(s)$ are small because $R(s)$ is small in $R(s)$-modules and the right adjoint, evaluation, commutes with infinite coproducts.  
One can show that $L^s R(s)$ is the $\mccR$-module with value $R(t)$ at $t \in \ccI$ if $\ccI(s,t)$ is non-empty and $0$ otherwise. Thus, if $\ccI(s,t)$ is non-empty there is a map $L^t R(t) \to L^sR(s)$ and $\mccR$ is the (finite) colimit over $\ccI^{{op}}$ of $L^sR(s)$.  

We use~\cite[19.9.1]{hh} to show that this colimit is a homotopy colimit.  This follows since $\ccI^{op}$ is a direct category and the diagram $L^sR(s)$ is Reedy cofibrant because all of the maps between objects $L^sR(s)$ and $L^tR(t)$ are either the inclusion of 0 or the identity map.  

To show that the finite homotopy colimit of small objects is small, we show that the adjoint finite limit commutes with infinite direct sums of abelian groups.  A finite product commutes with an infinite coproduct because finite products agree with finite coproducts here.  One can also check directly that equalizers commute with infinite coproducts for abelian groups.  Since finite limits are constructed from finite products and equalizers, this shows that finite limits commute with infinite coproducts.  
This argument is worked out in more detail for $\ccI$ a pullback diagram in Section 6.4 of~\cite{GScell}.

Next we compare $\bbarmccR$-modules and the diagram-injective structure on $\fmccR$-modules.   Since
$\bbarmccR$ is the inverse limit of $\fmccR$,  any
$\fmccR$-module $M$ defines an underlying $\ccI$ diagram of $\bbarmccR$-modules 
$\widetilde{M}$.  Denote the inverse limit of $\widetilde{M}$ by $\hat{M}$.   The functor $\fmccR \otimes_{\bbarmccR} -$
is left adjoint to this inverse limit functor and takes an $\bbarmccR$ module $N$ to
the $\fmccR$-module with $\fmccR(s) \otimes_{\bbarmccR} N$ at $s \in \ccI$.   
Since extension of scalars for a map
of ring spectra is a left Quillen functor and cofibrations and weak equivalences
are defined objectwise, $\fmccR \otimes_{\bbarmccR} -$ is a left Quillen functor from $\bbarmccR$-modules to $\fmccR$-modules.
We next apply the Cellularization Principle, Proposition~\ref{prop-cell-qe} (2),  
to this Quillen adjunction to induce a 
Quillen equivalence on the appropriate cellularizations.   Note that $\bbarmccR$
is cofibrant as an $\bbarmccR$-module and applying extension of scalars to it
gives $\fmccR$.   Since $\fmccR$ is diagram-injective fibrant as a diagram of ring spectra,
it is also  diagram-injective fibrant as an $\fmccR$-module.  Since
$\bbarmccR$ is the inverse limit of $\fmccR$,  cellularization induces a Quillen
equivalence.
\[ \mbox{$\bbarmccR$-cell-}\modcat{\bbarmccR} \simeq_{Q} \mbox{$\fmccR$-cell-}\modcat{\fmccR}  \]
Since $\bbarmccR$ is already a cofibrant generator of $\bbarmccR$-modules, the
cellular weak equivalences and fibrations in $\bbarmccR$-cell-${\bbarmccR}$-modules agree 
with those before cellularization.  Thus the cellularization of the model structure on the left 
is unnecessary and the statement follows.
\end{proof}

\begin{remark}
We want to point out that the model category
$\mbox{$\mccR$-cell-}\modcat{\mccR}$ is similar to the homotopy limit
homotopy theory considered in \cite{toen}, ~\cite{bergner1}
and~\cite{bergner2}.   However, in the present case we have shown this
is a model for the simpler category of $\bbarmccR$-modules.  

Also see Section~\ref{sec-inv-adjoint} for a reconsideration of the results here using the general results of Section~\ref{sec:adjoints}.
\end{remark}

\section{Adjunctions}\label{sec:adjoints}  
In this section we develop Quillen equivalences between categories of modules over diagrams of different shapes.  We consider the two basic cases corresponding to left and right adjoints as described in Section \ref{subsec:elemequivalences}.  We end by showing how these two base cases were combined in Proposition~\ref{prop-gen-pb}.
 
\subsection{The Left Adjoint Case}
Suppose $\bfE$ is a direct category with a fixed linear extension $d: \bfE \to \lambda$ for some ordinal $\lambda$. 
Let $i: \bfD \lra \bfE$ be an inclusion of a full subcategory and $\mccM: \bfE \lra \C$ be a diagram of
model categories and Quillen adjunctions as in Section~\ref{subsec:modcatdiag}
and~\ref{subsec:modstructure}.  Restriction of $\mccM$  produces a diagram $\mccM|_{\bfD}: \bfD \lra \C$,
and a restriction functor from $\mccM$-diagrams to $\MD$-diagrams
$$i^*: \modcat{\mccM}\lra \modcat{\mccM|_{\bfD}}. $$  
Assuming that each model category $\mccM(s)$ has all colimits, a left adjoint, $i_{*}$, of 
this restriction exists.  
Given an $\MD$-diagram $X$, we identify $i_{*}X(s)$ as in the end of Section~\ref{subsec:elemequivalences}.
Let $\bfD/{t}$ be the category of  morphisms $a:s \to t$ in $\bfE$ with domain $s$ in $\bfD$.
Then 
$$i_{*}X (t) = \colim_{s \in \bfD/{t}} a_{*}X(s).$$  
Since $\bfD$ is a full subcategory, for any $a': t \to t'$ in $\bfE$ the structure maps $a'_{*}i_* X(t) \to i_{*}X(t')$ can be filled in  by the universal property of colimits and the compatibility of compositions of arrows in $\bfE$.   

\begin{example}
Consider a category $\bfE$ with three objects and three non-identity morphisms 
$$\diagram
1\rto^{a} \drto_{c} & 2 \dto^{b}\\
& 3\\
\enddiagram$$ with $b\circ a = c$ and let $\bfD$ be the full subcategory $1 \xrightarrow{c} 3$.  Assume given a diagram of rings $R$ over $\bfE$ and an $R|_{\bfD}$-module $M$ with structure map $$c_{*}M(1) = R(3) \otimes_{R(1)} M(1) \to M(3).$$ Then $i_{*}M(2) = a_{*}M(1) =R(2) \otimes_{R(1)} M(1)$ and
the structure map $$b_{*}i_{*}M(2)= b_{*}a_{*}M(1) \to M(3)$$ agrees with the structure map $c_{*}M(1) \to M(3)$.
\end{example}
 
In our applications for~\cite{tnq3new}, the inclusion of $\bfD$ in
$\bfE$ is similar to this example and thus has a simplified left
adjoint.  This situation is described in the following statement.

\begin{prop}  Let $i: \bfD \lra \bfE$ be the inclusion of a full subcategory $\bfD$ in a direct category $\bfE$ such that $\bfD/{s}$ has a terminal object $a: t_{s} \to s$ for each $s \in \bfE$.  Given an $\MD$-diagram $X$, then the left adjoint to restriction evaluated at $s$ is 
$i_{*}X(s) =  a_{*}X(t_{s})$.  
\end{prop}

Since the restriction functor $i^*: \modcat{\mccM}\lra \modcat{\mccM|_{\bfD}}$  preserves objectwise weak equivalences and fibrations, it is a right Quillen functor on the diagram-projective model structures.  It then induces a Quillen equivalence on the cellularizations under the following conditions by the Cellularization Principle~\ref{prop-cell-qe}.

\begin{thm}  \label{thm5.3}
Let $\bfE$ be a direct category with $i: \bfD \lra \bfE$ an inclusion of a full subcategory and let $\mccM: \bfE \lra \C$ be a diagram of right proper, cellular, stable model categories 
such that each $\mccM(s)$ has all colimits.
There is a Quillen adjunction on the diagram-projective model structures.
$$
\adjunction{i_{*}}{\modcat{\MD}}{\modcat{\mccM}}{i^{*}}
$$  
\begin{enumerate}
\item Assume given a stable set of small cells $\cX$ in $\modcat{\MD}$ such that  $\underbar i_{*} X$ is small and the derived counit $X \to \underbar i^{*}\underbar i_{*} X $ is an equivalence for each $X \in \cX$.
Then $ (i_{*}, i^{*})$ induces a Quillen equivalence on the associated cellularizations.
$$
{\cX}\dcell\modcat{\MD}\simeq_{Q} {{\underbar i_{*}\cX}\dcell\modcat{\mccM}}
$$  
\item Assume given a stable set of small cells $\cY$ in $\modcat{\mccM}$ such that $\underbar i^{*}Y$  is small and  the derived counit $\underbar i_{*}\underbar i^{*} Y \to Y$ is an equivalence for each $Y \in \cY$.
Then $ (i_{*}, i^{*})$ induces a Quillen equivalence on the associated cellularizations.
$$
{\underbar i^{*}\cY}\dcell\modcat{\MD}\simeq_{Q}  {{\cY}\dcell\modcat{\mccM}}
$$  
\end{enumerate}
\end{thm}

\subsection{The Right Adjoint Case}
The right adjoint case is dual to the left adjoint case above; we spell out some of the details here.
Let $\bfE$ be an {\em inverse} category with a fixed linear extension $d: \bfE^{op} \to \lambda$ for some ordinal $\lambda$. 
Let $i: \bfD \lra \bfE$ be an inclusion of a full subcategory and $\mccM: \bfE \lra \C$ be a diagram of
model categories and Quillen adjunctions as in Section~\ref{subsec:modcatdiag}
and~\ref{subsec:modstructure}.  Restriction of $\mccM$  produces a diagram $\mccM|_{\bfD}: \bfD \lra \C$,
and a restriction functor from $\mccM$-diagrams to $\MD$-diagrams
$$i^*: \modcat{\mccM}\lra \modcat{\mccM|_{\bfD}}. $$  
Assuming that each model category $\mccM(i)$ has all limits, a right adjoint, $i_{!}$, of 
this restriction exists.  Given an $\MD$-diagram $X$, 
we identify $i_{!}X(t)$ as in the end of Section~\ref{subsec:elemequivalences}.
Let $t/\bfD$ be the category of  morphisms $a:t \to s$ in $\bfE$ with codomain $s$ in $\bfD$.
Then 
$$i_{!}X (t) = \ilim_{s \in t/\bfD} a_{*}X(s).$$    
Note that for $a': t' \to t$ in $\bfE$ the structure maps $i_{!}X(t') \to a'^{*}i_{!}X(t)$ can be filled in by the universal property of limits and the compatibility of compositions of arrows in $\bfE$ since $\bfD$ is a full subcategory. 

Since the restriction functor $i^*: \modcat{\mccM}\lra \modcat{\mccM|_{\bfD}}$  preserves objectwise weak equivalences and cofibrations, it is a left Quillen functor on the diagram-injective model structures.  It then induces a Quillen equivalence on the cellularizations under the following conditions by the Cellularization Principle~\ref{prop-cell-qe}.

\begin{thm}\label{thm5.4}
Let $\bfE$ be an inverse category with $i: \bfD \lra \bfE$ an inclusion of a full subcategory and let $\mccM: \bfE \lra \C$ be a diagram of right proper, cellular, stable model categories 
such that each $\mccM(s)$ has all limits.
There is a Quillen adjunction on the diagram-injective model structures.
$$
\adjunction{i^{*}}{\modcat{\mccM}}{\modcat{\MD}}{i_{!}}
$$  
\begin{enumerate}
\item Assume given a stable set of small cells $\cX$ in $\modcat{\mccM}$ such that $\underbar i^*X$ is small and the derived counit 
$X \to \underbar i_{!}\underbar i^{*} X $
 is an equivalence for each $X \in \cX$.
Then $ (i^{*}, i_{!})$ induces a Quillen equivalence on the associated cellularizations.
$$
{\cX}\dcell\modcat{\mccM}\simeq_{Q}  {{\underbar i^{*}\cX}\dcell\modcat{\MD}}
$$  
  
\item Assume given a stable set of small cells $\cY$ in $\modcat{\MD}$ such that $\underbar i^! Y$ is small and the derived counit 
$\underbar i^{*}\underbar i_{!} Y \to Y$
 is an equivalence for each $Y \in \cY$.
Then $ (i^{*}, i_{!})$ induces a Quillen equivalence on the associated cellularizations.
$$
{\underbar i_{!}\cY}\dcell\modcat{\mccM}\simeq_{Q}  {{\cY}\dcell\modcat{\MD}}
$$
\end{enumerate}
\end{thm}

\subsection{Inverse limit example revisited}\label{sec-inv-adjoint}
In Proposition~\ref{prop-gen-pb} a zig-zag of Quillen equivalences was used to produce a model for the category of modules over an inverse limit ring as the cellularization of a category of modules over the underlying diagram of rings.  We explain here that the second step in that zig-zag can be constructed as a combination of the left and right adjoint cases discussed above. 

As in Section \ref{sec-inverse}, assume given a finite, inverse category $\ccI$ with at most one morphism in each $\ccI(s,t)$ and a diagram of ring spectra, $\mccR$, indexed on 
$\ccI$.  Note that a finite, inverse category is also a direct category.
 Let $\ccIp$ be the category $\ccI$ with one added object $z$ and one
 morphism from $z$ to each object in $\ccI$ so that $z$ is an initial
 object in $\ccIp$.  Let $+$ denote the category with one object and
 one morphism. We next consider the right and left adjunction theorems above applied to the inclusions of $\ccI$ and $+$ into $\ccIp$. 

 Let $\fmccR$ be the fibrant replacement of $\mccR$  and let $\hat R$ denote its inverse limit as in 
 Section \ref{sec-inverse}. We extend the diagram $\fmccR$ on $\ccI$ to the diagram $\fRp$ on $\ccIp$ 
 such that $\fRp(z) = \hat R$.  By Theorem~\ref{thm5.4} inclusion $i: \ccI \to \ccIp$ induces a 
 Quillen adjunction $(i^*, i_!)$ on the diagram-injective model structures between $\fRp$ and $\fmccR$-
 modules.  Consider the diagram $\fmccR$ in $\fmccR$-modules and its image ${\mathit \underbar i_!} \fmccR$ 
 in $\fRp$-modules.  Since $\fmccR$ is fibrant, ${\mathit \underbar i_!} \fmccR (z) = \hat R$ and ${\mathit \underbar 
 i_!} \fmccR = \fRp$.  So Theorem~\ref{thm5.4}(2) implies the following.
  
 \begin{cor}\label{cor.1}
The adjoint functors $ (i^{*}, i_{!})$ induce a Quillen equivalence on the associated cellularizations.
$$
\fRp\dcell\modcat{\fRp}\simeq_{Q}  \fmccR\dcell\modcat{\fmccR}
$$
\end{cor}
 
 \begin{proof}
 The restriction ${\mathit \underbar i^*}{\mathit \underbar i_!} \fmccR = {\mathit \underbar i^*} \fRp$ is weakly equivalent to $\fmccR$.  The same argument that $\mccR$ is small in the proof of 
 Proposition~\ref{prop-gen-pb} also shows that $\fRp$ and $\fmccR$ are small in $\fRp$ and $\fmccR$-modules.  
\end{proof}

Next we consider the inclusion of $+$ in $\ccIp$ with image the object $z$.  Again we consider modules over the diagram of rings $\fRp$ on $\ccIp$.  Restricting to $+$, $\mathit \underbar i^* \fRp (z) = \hat R$.  
By Theorem~\ref{thm5.3} inclusion $i: + \to \ccIp$ induces a 
 Quillen adjunction $(i_*, i^*)$ on the diagram-projective model structures between $\hat R$ and $\fRp$-modules.  Next we note that $\mathit \underbar i_* \hat R (s)$ is weakly equivalent to $\fmccR (s)$ for any object $s$ in $\ccIp$.  So Theorem~\ref{thm5.3}(1) implies the following.
 
  \begin{cor}\label{cor.2}
The adjoint functors $ (i_{*}, i^{*})$ induce a Quillen equivalence on the associated cellularizations.
$$
\hat R\dcell\modcat{\hat R}\simeq_{Q}  \fRp\dcell\modcat{\fRp}
$$
\end{cor}

As in the end of the proof of Proposition~\ref{prop-gen-pb}, we note here that $\hat R$ already generates $\hat R$-modules so the cellularization on the left is unnecessary in Corollary~\ref{cor.2}.  Putting together Corollaries~\ref{cor.1} and \ref{cor.2}, and using the Quillen equivalence between the diagram-projective and diagram-injective model structures on $\fRp\dcell\modcat{\fRp}$, gives a zig-zag of Quillen equivalences between ${\hat R}$-modules and $ \fmccR\dcell{\fmccR}$-modules. As in the proof of Proposition~\ref{prop-gen-pb}, Lemma~\ref{lem.ext.scalars} and Corollary~\ref{cor-cell-qe} then show that $\fmccR$ can be replaced by $\mccR$.

\begin{remark}
One could consider a given general diagram $\mccR$ on $\ccI$ instead of $\fmccR$ and proceed as in Corollary~\ref{cor.1}.  To extend this to a diagram on $\ccIp$ one would need to use the inverse limit $\mccR'$ of $\mccR$ at the object $z$ in $\ccIp$.  The derived functor $\mathit \underline i_!$ evaluated at $z$ applied to $\mccR$ would be $\hat R$, the homotopy limit of $\mccR$ or the inverse limit of a fibrant replacement of $\mccR$.    Here the limit, $\mccR'$, and the homotopy limit, $\hat R$, are not weakly equivalent in general. This would cause problems for the steps in Corollary~\ref{cor.2}, because using restriction would force one to work over $\mccR'$-modules.
\end{remark}

\begin{appendix}


\section{Cellularization of model categories}\label{app:cell}
 
The Quillen equivalences developed in the body of this paper 
rely on the process of cellularization (also known as right
localization or colocalization) of model categories from~\cite{hh}. In~\cite{GScell}, we show that Quillen adjunctions between stable model categories induce Quillen equivalences between their respective cellularizations provided there is an equivalence on the chosen cells.  

Here we always consider stable cellularizations of stable
model categories.  We say a set $\cK$ is {\em stable} if it is closed under suspension and desuspension up to weak equivalence. Given a stable model category $\mccM$ and $\cK$ a stable set of  objects, $\cK$-cell-$\mccM$ is again a stable model category by~\cite[4.6]{BR.stable}. 
We say an object $K$ is small in
the homotopy category (simply {\em small} elsewhere in the paper)
if,  for any set of objects
$\{Y_{\alpha}\}$, the natural map $\bigoplus_{\alpha} [K,
Y_{\alpha}]\lra [K,\bigvee_{\alpha}Y_{\alpha}]$ is an isomorphism. 

\begin{prop} \cite{GScell}
\label{prop-cell-qe}
Let $\mccM$ and $\cN$ be  stable, cellular, right proper model
categories with $F: \mccM \to \cN$ a left Quillen 
functor with right adjoint $U$.   
Denote the associated derived functors by $\underbar F$ and $\underbar U$.
\begin{enumerate}
\item Given $\cK= \{ A_{\alpha} \}$ a stable set of small objects in $\mccM$, let $\underbar F\cK = \{ \underbar F A_{\alpha}\}$ be the corresponding
set of objects in $\cN$.  If for each $A_{\alpha}$ the image $\underbar F A_{\alpha}$ is small and $A_{\alpha} \to \underbar U \underbar F A_{\alpha}$ is a weak equivalence,
then the $\cK$-cellularization
of $\mccM$ and the $\underbar F\cK$-cellularization of $\cN$  are Quillen equivalent: 
\[ \mbox{$\cK$-cell-$\mccM$}\simeq_Q \mbox{$\underbar F\cK$-cell-$\cN$} \]
\item Given $\mcL = \{B_{\beta}\}$ a stable set of small objects in $\cN$, let $\underbar U\mcL = \{\underbar U B_{\beta}\}$ be the corresponding
set of objects in $\cN$.  If for each $B_{\beta}$ the image $\underbar U B_{\beta}$ is small and  $\underbar F \underbar U B_{\beta} \to B_{\beta} $ is a weak equivalence,
 then the $\mcL$-cellularization
of $\cN$ and the $\underbar U\mcL$-cellularization of $\mccM$
 are Quillen equivalent: 
\[ \mbox{$\underbar U\mcL$-cell-$\mccM$} \simeq_{Q} \mbox{$\mcL$-cell-$\cN$}\]
\end{enumerate} 
\end{prop}

If $F$ and $U$ induce a Quillen equivalence on the original categories,
then the hypotheses in Proposition \ref{prop-cell-qe} are automatically
satisfied.   Thus, the cellularizations are Quillen equivalent. 
 
\begin{cor}\cite{GScell}\label{cor-cell-qe}
Let $\mccM$ and $\cN$ be stable, 
cellular, right proper model
categories with $F: \mccM \to \cN$ a Quillen 
equivalence with right adjoint $U$.  
Denote the associated derived functors by $\underbar F$ and $\underbar U$.
\begin{enumerate}
\item Given $\cK= \{ A_{\alpha} \}$ a stable set of small objects in $\mccM$, let $\underbar F\cK = \{ \underbar F A_{\alpha}\}$ be the corresponding
set of objects in $\cN$. Then the $\cK$-cellularization
of $\mccM$ and the $\underbar F\cK$-cellularization of $\cN$  are Quillen equivalent: 
\[ \mbox{$\cK$-cell-$\mccM$}\simeq_Q \mbox{$\underbar F\cK$-cell-$\cN$} \]
\item Given $\mcL = \{B_{\beta}\}$ a stable set of small objects in $\cN$, let $\underbar U\mcL = \{\underbar U B_{\beta}\}$ be the corresponding
set of objects in $\cN$.  Then the $\mcL$-cellularization
of $\cN$ and the $\underbar U\mcL$-cellularization of $\mccM$
 are Quillen equivalent: 
\[ \mbox{$\underbar U\mcL$-cell-$\mccM$} \simeq_{Q} \mbox{$\mcL$-cell-$\cN$}\]
\end{enumerate} 
\end{cor}

This corollary also follows from the dual of~\cite[2.3]{hovey-stable}. 
\end{appendix}


\end{document}